\documentclass[11pt]{amsart}

\usepackage{amsrefs}
\usepackage{epsfig}
\usepackage{psfrag, graphics}
\usepackage{color}

\usepackage{comment}

\usepackage{mathtools}

\newtheorem{theorem}{Theorem}[section]
\newtheorem{corollary}[theorem]{Corollary}
\newtheorem{proposition}[theorem]{Proposition}
\newtheorem{lemma}[theorem]{Lemma}

\theoremstyle{definition}

\newtheorem{remark}[theorem]{Remark}

\newcommand{\infline}{\mathbb{C}P^1(\infty)}
\newcommand{\highdeg}{S}

\def\R{\mathbb R}
\def\C{\mathbb C}

\def\Z{\mathbb Z}

\newcommand{\CC}{{\mathbb C}}

\newcommand{\RR}{{\mathbb R}}
\newcommand{\ZZ}{{\mathbb Z}}
\newcommand{\cal}{\mathcal}
\newcommand{\Ind}{\mathrm{index}}
\newcommand{\CZ}{\operatorname{CZ}}
\newcommand{\e}{\operatorname{e}}

\begin{document}

\title{Symplectic embeddings of polydisks}

\author{R. Hind}
\author{S. Lisi}

\date{\today}

\maketitle

\section{Introduction and notation}

In this note, we obtain new obstructions to symplectic embeddings of a
product of disks (a polydisk) into a 4-dimensional ball.

Let us equip $\R^4$ with coordinates $x_1,y_1,x_2,y_2$ and its standard symplectic form $\omega = dx_1 \wedge dy_1 + dx_2 \wedge dy_2$. The polydisk $P(r,s)$ is defined to be $$P(r,s) = \{ \pi(x_1^2 + y_1^2) \le r, \pi(x_2^2 + y_2^2) \le s\}$$ with symplectic form induced from $\R^4$. The ball of capacity $a$ (i.e.~ the ball of radius
$\sqrt{\frac{a}{\pi}}$) is $$B(a) = \{\pi( x_1^2 + y_1^2 + x_2^2 + y_2^2) \le a\}$$ again with the induced symplectic form.

Given $r,s$, the embedding problem for the polydisk $P(r,s)$ into a ball is to find the infimum of the set of $a$ such that there exists a symplectic embedding $P(r,s) \to B(a)$.

There are two well known sequences of symplectic capacities giving obstructions to symplectic embeddings in 4 dimensions, namely the Ekeland-Hofer capacities \cite{ekehof} (which are actually defined in all dimensions) and the Embedded Contact Homology (ECH) capacities introduced by Hutchings \cite{HutchingsCapacities}. In the case when $(r,s)=(1,2)$ both the Ekeland-Hofer and ECH capacities give the obstruction $a \ge 2$. As $P(1,2)$ and $B(2)$ have the same volume this restriction also follows since symplectic embeddings preserve volume.

Our main theorem solves this embedding problem in the case when $(r,s)=(1,2)$.

\begin{theorem} \label{main}
	There exists a symplectic embedding
	\[ P(1,2) \coloneqq D(1) \times D(2) \hookrightarrow B( a )\]
if and only if $a \ge 3$.
\end{theorem}

The sufficiency of $a \ge 3$ follows since by inclusion $P(1,2) \subset B(3)$. In general we have $P(r,s) \subset B(r+s)$. Hence the theorem can be rephrased as saying that for $P(1,2)$ the inclusion map gives the optimal embedding. The necessity of $a \ge 3$ implies that for this particular embedding problem neither the Ekeland-Hofer nor ECH capacities give a sharp obstruction. We contrast this with the case of ellipsoid embeddings into a ball when the ECH capacities give a complete list of obstructions, see \cite{mcduff}.
Our obstruction does not
come from a symplectic capacity, but instead from pseudoholomorphic
foliations, thus the techniques seem to be special to dimension $4$.

The embedding problem is discussed at length by Schlenk in the book \cite{SchlenkBook}, and in particular the technique of symplectic folding. One consequence of this is the following.

\begin{theorem} \label{felix} (\cite{SchlenkBook}, Proposition $4.3.9$.)
If $s>2$ then there exists a symplectic embedding $P(1,s) \to B(a)$ for all $a>2+\frac{s}{2}$.
\end{theorem}

When $s>2$ we have $2+\frac{s}{2} < 1+s$ and so this result implies that for $P(1,s)$ with $s>2$ the inclusion map is never optimal.

Finally we remark that when $s$ increases even the embeddings in Theorem \ref{felix} are far from optimal. Indeed, Theorem $3$ in \cite{SchlenkBook} implies that by taking $s$ sufficiently large we can find embeddings $P(1,s) \to B(a)$ whose image occupies an arbitrarily large proportion of the volume.


\subsection{Organization and outline of the proof.} \label{organize}

Define $\CC P^2(R)$ to be
complex projective space equipped with the Fubini-Study symplectic form
scaled so the area of the line is $R$, and denote by $\CC
P^1(\infty)$ the line at infinity. Then $\CC P^2(R)$ is the union of the open ball of capacity $R$ and $\CC
P^1(\infty)$. Our goal is to show that for any $R < 3$,
there does not exist an embedding $P(1,2) \hookrightarrow \CC P^2(R) \setminus \CC
P^1(\infty)$.

Our proof will be by contradiction.  Assume such an embedding exists.
Let $L$ be the Lagrangian torus at the ``corner'' of the bidisk:
\[
	L = \partial D(1) \times \partial D(2)
	\subset P(1,2) \subset \CC P^2(R) \setminus \CC P^1(\infty)
	\subset \CC P^2(R).
\]
Note that the (closed) ball of capacity $1$, centred at $(0,0)$
sits in the (closed) bidisk $D(1) \times D(2)$
and does not intersect $L$. Let $X$ be the symplectic manifold obtained
by blowing up this ball of capacity $1$:
$X = \CC P^2 (R) \sharp \overline{\CC P^2(1)}$. Let $\omega_R$ be the symplectic form on $X$.
Denote the exceptional divisor by $E$.  Note that $E$ is a symplectic
sphere of area $1$ and that $X$ is the non-trivial $\CC P^1$ bundle
over $\CC P^1$.

In section \ref{closed} we describe two classes of holomorphic curves in $X$, namely curves in the class of a fiber which foliate $X$, and sections $S$ of high degree $d = S \bullet \CC P^1(\infty)$. Such curves exist for any tame almost-complex structure on $X$.

Next we will stretch
the neck along the boundary of a small tubular neighbourhood of $L$ as described in \cite{BEHWZ}, and take limits of the closed curves described in section \ref{closed}.
The stretching is described in section \ref{indices} together with the Fredholm theory for the limiting curves.

We will also be interested in intersection properties of the limiting finite energy curves. Therefore we briefly describe Siefring's extended intersection number in section \ref{siefring}.

With all of this in place,
the limits of curves in the fiber class are considered in section \ref{fefoln}. They form a finite energy foliation, see \cite{hwzfol}.

Finally in section \ref{highdegree} we consider the limit of a high degree curves with
pointwise constraints (the degree will depend on $R$). We are able to make deductions about this limit based on the structure of the finite energy foliation, and we will see that for the limit to have nonnegative area it is necessary that $a \ge 3$, thus proving Theorem \ref{main}.

To fix notation, let $(k,l) \in H_1(L)$ denote the homology class
$k[\partial D(1)] + l[\partial D(2)]$.

\section{Closed curves in $X$} \label{closed}

We study two classes of closed curves in $X = \CC P^2 (R) \sharp \overline{\CC P^2(1)}$, the blow-up of a ball of capacity $1$ in $\CC P^2$ with its Fubini-Study form scaled such that lines have area $R$. We work with tame almost-complex structures $J$ on $X$ for which the exceptional divisor $E$ and the line at infinity $\CC P^1(\infty)$ are complex. Recall from section \ref{organize} that the blown-up ball is located inside our polydisk and so disjoint from $\CC P^1(\infty)$.

The results are fairly standard and so we merely outline the proofs.

%
\subsection*{Fiber curves}

The result here is the following:

\begin{proposition} \label{fiberfoln} (see \cite{hindivrii}, Proposition $4.1$)
For a generic $J$, the manifold $X$ is foliated by $J$ holomorphic spheres
in the class $[\CC P^1(\infty)] - [E]$.
\end{proposition}

{\it Outline of the proof.}

Let $C$ be a holomorphic sphere in the class  $[\CC P^1(\infty)] - [E]$. Then $C \bullet C =0$. The index
formula, see for example \cite{msa}*{Theorem 3.1.5} gives the virtual dimension $\mathrm{index}(C)=2$. By positivity
of intersection such a curve $C$ must intersect $\CC P^1(\infty)$ transversally in a single point and so be somewhere injective. Further, by
the adjunction formula \cite{msa}*{Section 2.6}, $C$ is embedded. Then
by automatic regularity, see \cite{msa}*{Lemma 3.3.3}, the corresponding moduli space has dimension $2$.

Suppose we have a nodal curve representing the same homology class. As
$R<3$ and $\mathrm{area}(C)=R-1$, by area considerations a nodal curve
can only have two two components, one in the class $[ \CC P^1] - 2[E]$
and the other in the class $[E]$.  Note that each of these components
has index $0$ and is somewhere injective, and thus transverse. They have
intersection number $2$, and for generic $J$, they intersect transversely
and thus in $2$ points.  Gluing would then give a somewhere injective
index $2$ curve in our class with one geometric intersection.  This is
a contradiction to the adjunction formula.

Hence the moduli space of curves in this class is compact and as $C
\bullet C =0$ consists of curves with disjoint image. It follows that
the curves foliate $X$ as required.
\qed

%
\subsection*{High degree curves}

We now consider holomorphic spheres $\highdeg$ in $X$ of high degree $d >1$. To work with isolated curves we impose a collection of point constraints and then have the following.

\begin{proposition} \label{highdeg}
Given $2d$ constraint points and a generic $J$ (to exclude nodal curves), there
is a unique embedded holomorphic sphere $\highdeg$ in the class $d[\CC P^1(\infty)] - (d-1)[E]$ passing through the points.
\end{proposition}

{\it Outline of the proof.}

Let us fix $2d$ points in $X$, and let $\highdeg$ be a holomorphic sphere in the class $d[\CC P^1(\infty)] - (d-1)[E]$. If the sphere is smooth (i.e.~not a nodal curve), it is somewhere injective because $d$ and $d-1$ are
coprime. By the adjunction formula as above, the curve is then embedded.

As $\mathrm{index}(\highdeg)=2d$ the virtual
dimension of curves in our class passing through the fixed points is $0$,
and as $\highdeg \bullet \highdeg = 2d-1$
there can be at most one such curve. Further,
an index calculation shows that nodal representatives of this homology
class are of codimension at least $2$ and so we do not expect such a
nodal curve to intersect all $2d$ points for generic $J$.

Finally we are left to check that the moduli space is nonempty. We observe that if the moduli space is nonempty for {\it some} $J$ then the
corresponding Gromov-Witten invariant is $\pm 1$ and so it will be nonempty for all generic $J$. Hence we can choose
a complex structure such that the projection $X \to \CC P^1(\infty)$
along the fibers of the above foliation is holomorphic, describing $X$
as a nontrivial holomorphic $\CC P^1$ bundle over $\CC P^1$.  Our curves
are sections of the above bundle with $d$ `zeros' (intersections with $\CC
P^1 (\infty)$) and $d-1$ `poles' (intersections with $E$) and exist because the bundle has degree $1$. 
\qed

\section{Fredholm theory} \label{indices}

In this section, we describe our neck stretching procedure and determine the Fredholm index of the types of curve
that can appear in the holomorphic buildings we obtain by stretching
the neck.

\subsection{Stretching the neck.} \label{stretch}

To perform a neck stretch we must choose a tubular neighborhood $U$ of
$L$ in $X$. Recall that $L = \partial D(1) \times \partial D(2)
	\subset P(1,2)\sharp \overline{\CC P^2(1)} \subset X$. By Weinstein's theorem a tubular neighborhood can be identified
with a neighborhood of the zero-section in $T^* L$. We fix a flat metric
on $L$ and let our neighborhood be a unit cotangent disk bundle with
boundary $\Sigma = T^3$. The Louville form on $T^* L$ restricts to a contact
form $\lambda$ on $\Sigma$ and the associated Reeb vector field $v$ generates the
geodesic flow. We can identify a neighborhood of $\Sigma$ in $X$ with the symplectic manifold $(\Sigma \times (-\epsilon,\epsilon), d(e^t \lambda))$, where $\lambda$ is pulled back from $\Sigma$ to $\Sigma \times (-\epsilon,\epsilon)$ using the natural projection.

Now choose a sequence of almost-complex structures $J^N$
on $\CC P^2(R)$.
We arrange that these all coincide outside of a
neighborhood of $\Sigma$ and also that both $E$ and $\CC P^1(\infty)$
are $J^N$--holomorphic. Further, on the neighborhood $\Sigma \times (-\epsilon,\epsilon)$ we assume that all $J^N$ preserve the planes $\xi = \{\lambda =0\}$ and coincide on these planes. However we require $J^N(v) = \frac{1}{N} \frac{\partial}{\partial t}$. Then $(\CC P^2, J^N)$ admits a biholomorphic
embedding of $\Sigma \times (-N,N)$ equipped with a translation invariant
complex structure mapping the Reeb vectors to the unit vectors in
$(-N,N)$.

A sequence of $J^N$ holomorphic spheres of fixed degree has
a limit in the sense of \cite{BEHWZ}. This is a holomorphic building
whose components are finite energy curves mapping into three almost-complex
manifolds with cylindrical ends, namely
\begin{enumerate}
\item $X \setminus U \cup \Sigma \times (-\infty,0]$;
\item $\Sigma \times \RR$;
\item $U \cup \Sigma \times [0,\infty)$.
\end{enumerate}
The almost-complex structures here coincide with the $J^N$ away from the ends, and are translation invariant on the cylindrical ends.
We note that given a compact set $K$ in any of these three cylindrical manifolds there exists a biholomorphic embedding $K \to (X, J^N)$ for all $N$ sufficiently large. For the theory of finite energy curves see 
\citelist{\cite{hofa} \cite{hofi} \cite{hoff}}. 
We will often use diffeomorphisms to identify the first and third cylindrical manifolds with $X \setminus L$ and $T^* L$ respectively.

The finite energy curves constituting our holomorphic building have a level structure. If the building has level $k$ then we say that the curves mapping to $T^* L$ have level $0$, the curves mapping to $X \setminus L$ have level $k$ and the curves mapping to $\Sigma \times \RR$ have a level $l$ with $0 < l < k$. Finite energy curves of level $0$ have only positive ends, curves of level $k$ have only negative ends, and the punctures of curves mapping to $\Sigma \times \RR$ are positive if the $t$ coordinate approaches $+\infty$ as we approach the puncture and negative if the $t$ coordinate approaches $-\infty$. For a union of the finite energy curves to form a building we require that the punctures be matched in pairs such that each positive end of a curve of level $l$ is matched with a negative end of a curve of level $l+1$ asymptotic to the same geodesic. By identifying the underlying Riemann surfaces (that is, the domains of the finite energy curves) at matching punctures we can think of the domain of a holomorphic building as a nodal Riemann surface.

\subsection{Compactification and area.} \label{area}

The finite energy curves in all three manifolds have punctures asymptotic to Reeb orbits on $\Sigma = T^3$. Under our identification of the first manifold with $X \setminus L$ these curves can be extended to maps from the oriented blow-up of the Riemann surface at its punctures, mapping the boundary circles to the asymptotic closed Reeb orbit on $L$, see Proposition $5.10$ of \cite{BEHWZ}. Away from their singularities the finite energy curves are now symplectic immersions, that is, the pull back of the symplectic form on $X$ is nondegenerate. The integral of this symplectic form is defined to be the {\it area} of the finite energy curve. As $L$ is Lagrangian, it follows from the compactness theorem of \cite{BEHWZ} that for a converging sequence of $J^N$ holomorphic spheres in a fixed homology class $A \in H_2(X, \ZZ)$, the sum of the areas of the components of the limiting building in $X \setminus L$ is equal to $\omega_R(A)$.

\subsection{Fredholm indices.}

By our construction, the Reeb dynamics on $\Sigma = T^3$ gives
the geodesic flow on $L=T^2$ for the flat metric.
The Reeb orbits on $\Sigma = T^3$ are then in bijective correspondence
with the geodesics on $T^2$.
There is an $S^1$ family of closed geodesics in each homology class $(k,l)$
for $(k,l) \in \ZZ^2 \setminus \{ (0,0) \}$. Here we recall from section \ref{organize} that $L$ lies in the boundary of a polydisk $P(1,2)$ and  $(k,l) \in H_1(L)$ denotes the homology class
$k[\partial D(1)] + l[\partial D(2)]$

We make the convention that a curve in $X \setminus L$ is of degree $d$
if its intersection number with $\CC P^1(\infty)$ is $d$.  Denote its
intersection number with $E$ by $e$.  Observe that since these are
intersections with classes that have holomorphic representatives, $d
\ge 0$ and $e \ge 0$, unless the holomorphic curve is a cover of $E$ (in which case, $e < 0$ and $d = 0$).

If the holomorphic curve in question is somewhere injective, then for
a generic almost complex structure, the moduli space of nearby finite energy
curves will have dimension given by its Fredholm index, and in the remainder of this section we list these indices.

We start with curves in $X \setminus L$.


\begin{proposition} \label{out}

Let $C$ be a curve in $X \setminus L$ of degree $d$, with $e$
intersections with $E$,  and with $s$ negative ends asymptotic to
geodesics in the classes $(k_i,l_i)$ respectively for $1 \le i \le s$.

The index of $C$ (as an unparametrized curve, allowing
the asymptotic ends to move in the corresponding $S^1$ family of Reeb
orbits) is given by
$$\mathrm{index}(C) = s - 2 +6d -2e + 2\sum_{i=1}^s (k_i+l_i).$$

\qed
\end{proposition}

For curves in $\RR \times \Sigma$, we have
\begin{proposition} \label{neck}
The index of a finite energy curve $C$ in
$\RR \times \Sigma$ of genus $0$ with
$s^+$ positive punctures and $s^-$ negative punctures is given by
$$\mathrm{index}(C)=2s^+ + s^- -2.$$
\end{proposition}

Finally, for curves in $T^* L$ we obtain,
\begin{proposition} \label{in}
The  index of a genus $0$, finite energy curve
$C$ in $T^*T^2$, with $s$ (positive) punctures, is
$$\mathrm{index}(C)=2s-2.$$
\end{proposition}

Note that we may consider a (sub)building consisting of a level in $T^*L$ and
possibly multiple levels in $\R \times T^3$ with the requirement of matching ends. If we now interpret $s$ to be the
number of positive punctures of the uppermost level of the building, the
index formula of Proposition \ref{in} also applies, giving the deformation index of the space of buildings.


Each of these index calculations follows immediately as an application of
the appropriate version of Riemann-Roch  and
by the computation  of the Conley-Zehnder
indices of the Reeb orbits in $T^3$, where the computation is done with
respect to the trivialization induced by $T^*T^2$ for Propositions
\ref{neck} and \ref{in}, and with respect to the trivialization from
the bidisk $D^2 \times D^2$ for Proposition \ref{out} (see Appendix
\ref{A:RiemannRoch}).

\section{Intersections of finite energy curves} \label{siefring}

Siefring's intersection theory for punctured pseudoholomorphic
curves, extended by Siefring and Wendl to the Morse-Bott case
\citelist{\cite{Siefring} \cite{SiefringWendl} \cite{wendl}*{Section 4.1}}
is very useful to our study and we briefly outline this here.

We require the existence of the following generalization of the intersection
number, valid for curves with cylindrical ends. We say that a curve has
cylindrical ends if it can be parametrized by a punctured Riemann surface
and is pseudoholomorphic of finite energy
in a neighbourhood of the punctures.
\begin{theorem}
	[\citelist{\cite{Siefring}*{Theorem 2.1} \cite{wendl}*{Section
			4.1}}]
	\label{T:intersectionNumber}
  Let $(W, \omega, J)$
  be an almost complex manifold with cylindrical ends that are
  symplectizations of contact manifolds,
  with almost complex structure adapted to Morse-Bott contact forms.

  Then, for two curves with cylindrical ends, $u$, $v$ in $W$,
  there exists an extended intersection number $u \star v$
  with the following properties:
  \begin{enumerate}
    \item If $u$ and $v$ asymptote to disjoint sets of orbits, $u \star v$
      is the intersection number $u \bullet v$.
    \item The extended intersection number
	    $u \star v$ is invariant under homotopies of $u$ and $v$ through
	    curves with cylindrical ends.
    \item If $u$ and $v$ are geometrically distinct (i.e.~their images do not
	    coincide on any open set), $u \star v \ge 0$.
    \end{enumerate}
    Furthermore, if a sequence of curves representing classes $C$ and $C'$ in
    a compact manifold $X$ converge to holomorphic buildings with levels $u_1,
    \dots, u_N$ and levels $u'_1, \dots, u'_N$, the following inequality holds:
    $C \bullet C' \ge \sum u_i \star u'_i$.
\end{theorem}
In our setting, the existence of this intersection
number can be obtained by hand. Specifically, in $X \setminus L$,
the intersection number $u \star v$ of two curves $u$ and $v$ is obtained
by pushing $u$ off of itself with a small perturbation that is
tangent to the Morse-Bott family of orbits at each puncture to obtain a
curve $u'$. Then take $u \star v = u' \bullet v$.

A key fact necessary for the existence of this extended
intersection number, but also for its computation,
is the asymptotic convergence formula for pseudoholomorphic
curves. We will use this later, so we recall it here for
the benefit of the reader. The statement we provide here
is an immediate corollary of the more general results of Hofer, Wysocki
and Zehnder and of Siefring.

First, we recall the definition of an asymptotic operator associated
to a closed Reeb orbit.  Let $(M, \ker \alpha)$ be a contact manifold
with preferred contact form $\alpha$, and let $R$ be the associated
Reeb vector field. Let $J$ be an almost complex structure on $\xi = \ker \alpha$
compatible with $d\alpha|_{\xi}$.
Let $\gamma \colon \R/ T\Z \to M$ be a closed Reeb orbit of period $T$ (not
necessarily the minimal period). Choose a torsion-free connection $\nabla$.
Then, the asymptotic operator at the orbit $\gamma$ is the unbounded
self-adjoint operator:
\begin{align*}
	A_{\gamma, J} \colon L^2( \gamma^*\xi) &\to L^2( \gamma^*\xi) \\
	\eta &\mapsto -J( \nabla_{\partial_t} \eta - T \nabla_\eta R ).
\end{align*}
Note that this differential operator does not depend on the choice of
connection, but does depend on choice of almost complex structure.

We state the asymptotic convergence theorem for the case of a negative
puncture, since this is the case we need. A positive puncture behaves
analogously, with the signs modified appropriately.
\begin{theorem}
	[\citelist{ \cite{hofa}*{Theorem 1.3}
		\cite{SiefringAsymptotics}*{Theorem 2.2 and Theorem 2.3}} ]
	\label{T:asymptotic}

  Let $(W, \omega, J)$ be an almost complex manifold with cylindrical
  ends that are symplectizations of contact manifolds, with almost
  complex structure adapted to Morse-Bott contact forms.

  Let $u_1$ and $u_2$ be finite energy punctured pseudoholomorphic
  half-cylinders $u_1, u_2 \colon (-\infty, 0] \times S^1 \to X$,
  asymptotic to the same (possibly multiply covered) orbit $\gamma$
  of period $T > 0$.

  Then, for $R << 0$, there exist proper embeddings
  $\psi_1, \psi_2 \colon (-\infty, R] \times S^1 \to (-\infty, 0] \times S^1$,
  asymptotic to the identity,
  maps $U_i \colon (-\infty, R] \times S^1 \to (\gamma)^*\xi$ with
  $U_i(s,t) \in \xi(\gamma(Tt))$ so that
\[
	u_i( \psi_i(s,t)) = (-Ts, \exp_{\gamma(T t)} U_i(s,t)).
\]

  Furthermore, either $U_1 - U_2$ vanishes identically, or there exists
  a positive eigenvalue $\lambda > 0$ of the asymptotic operator
  $A_{\gamma, J}$ and a corresponding eigenvector so that
  \[
	  U_1(s,t) - U_2(s,t) = \e^{\lambda s} ( e(t) + r(s,t)),
  \]
  where $r$ and all its derivatives decay exponentially fast.
\end{theorem}
In particular, we will apply this theorem in the case where the two negative
ends come from curves asymptotic to different covers of the same simple orbit.
In that case, we can reparametrize the ends by factoring through a covering
map so that both converge to the same (higher multiplicity) orbit.

%
%
%
%

\section{A finite energy foliation} \label{fefoln}

Recalling Proposition \ref{fiberfoln}, we may assume that for each $N< \infty$ there
exists a foliation of $X$ by $J_N$ holomorphic spheres in the class of
a fiber $F= [\CC P^1(\infty)] - [E]$ of the bundle $X \to E$.

We will follow the method of \cite{hwzfol} to construct a foliation of
$X \setminus L$ by finite energy curves.  To this end,
we fix a countable dense set of points $\{p_i\} \in X$.  For each $N$
there exists a unique curve $C^N_i$ in the fiber class intersecting $p_i$.
Taking a diagonal subsequence of $N \to \infty$ we may assume that all
$C^N_i$ converge as $N \to \infty$ to holomorphic buildings having
a component $C_i$ passing through $p_i$. Positivity of intersection
implies that these curves are either disjoint or have identical image.
By taking further limits of the $C_i$ we obtain a finite energy foliation
${\cal F}$ of $X \setminus L$, see \cite{hwzfol}. We describe this foliation in section \ref{folnxl}. The limits of the $C^N_i$ also have components in $T^* L$ and $\R \times \Sigma$. These curves are described in section \ref{tls}.


\subsection{Limit curves in $X \setminus L$} \label{folnxl}

\begin{proposition}\label{foliationF}
  The leaves of ${\cal F}$ in $X \setminus L$
	consist of three kinds of curves:
\begin{enumerate}
\item Embedded closed spheres in the fiber class.
\item Planes $C_0$ of degree $0$ asymptotic to $(1,0)$ geodesics.
\item Planes $C_1$ of degree $d=1$ with $e = C_1 \bullet E =1$
	and asymptotic to $(-1,0)$ geodesics.
\end{enumerate}
These leaves foliate $X \setminus L$.
\end{proposition}

\begin{proof}
We pick a point $p \in X \setminus L$ and look at a limit of
$J^N$-holomorphic spheres $C^N$ through $p$. The component of the limit
through $p$ is the leaf of ${\cal F}$ through $p$.  If $C^N$ converges to
a closed curve then it is of the first type. Nodal curves in $X \setminus L$ are excluded as in section \ref{closed}.

Now suppose that the limit is a (non-trivial) holomorphic building, $B$.
Observe that the symplectic area of each of the $C^N$ is $R-1$, and thus by our assumption that $R<3$ the
sum of the areas of the components in $X \setminus L$ is $R-1<2$ (as symplectic area is preserved in the limit, see the discussion in section \ref{area}).
Also recall that for any curve $C$ in $X \setminus L$, of degree $d$ and
intersection number $e$ with the exceptional divisor,
and negative ends asymptotic to geodesics in classes $(k_i, l_i)$, we can compactify to a symplectic surface with boundary on the relevant geodesics (see again section \ref{area}) and then
the area is given by
\begin{equation} \label{areafor}
\mathrm{area}\:(C) = Rd - e + \sum (k_i + 2 l_i).
\end{equation}

By positivity of intersection exactly curve of $B$ in $X \setminus L$ has degree $1$.
By formula (\ref{areafor}) the degree $0$ components have positive
integral area, and so as the total area is $R-1 <2$ there can be at most one of these. Thus, for the building $B$ to be non-trivial
(i.e.~not a closed curve)
there must be exactly two components in $X \setminus L$,
one of area $1$ and the other of degree $1$ and area $R-2$.
Also by our area bounds we see that neither component can contain a
closed curve. As $B$ is a limit of curves of genus $0$, after identifying matching ends the building itself must have genus $0$.
Thus, as there are no finite energy planes in $T^* L$ (as there are no
contractible geodesics) we conclude that that two components must both be planes, that is, have a single negative end.

We denote these two planes by $C_0$ and $C_1$,
where $C_0$ is of degree $0$ and asymptotic to a $(k,l)$
geodesic and $C_1$ is of degree $1$. As $C_0$ and $C_1$ are components of a holomorphic building whose other components lie in $T^* L$ we necessarily have that $C_1$ is asymptotic to a $(-k,-l)$
geodesic.

Note that by the area condition, $C_0$ must be somewhere injective
and that $C_1$ must be somewhere injective since it is of degree $1$.

By positivity of intersection exactly one of $C_0$ or $C_1$ can intersect the exceptional divisor $E$ and so we reduce to two cases.

{\bf Case $1$: $C_0 \bullet E =0$; $C_1 \bullet E =1$.}

Since the curves are somewhere injective, their Fredholm indices must be
non-negative. Therefore
\begin{alignat*}{1}
\mathrm{index}(C_0) &= -1 + 2(k+l) 	\ge 0 \\
\mathrm{area}(C_0)  &= k+2l 		=1. \\
\mathrm{index}(C_1) &= 3 - 2(k+l) 	\ge 0 \\
\mathrm{area}(C_1)  &= R-1-(k+2l) 	= R-2.
\end{alignat*}
Solving these inequalities gives $k=1$, $l=0$ and we get curves of types $(2)$ and $(3)$.

{\bf Case $2$: $C_0 \bullet E =1$; $C_1 \bullet E =0$.}

Now we compute
\begin{align*}
\mathrm{index}(C_0) &= -3 + 2(k+l) \ge 0\\
\mathrm{area}(C_0)  &=k+2l-1=1. \\
\mathrm{index}(C_1) &= 5 - 2(k+l) \ge 0\\
\mathrm{area}(C_1)  &=R-(k+2l)=R-2.
\end{align*}

Solving these we get $k=2$, $l=0$. Note that $C_0$ is somewhere injective.
If $C_0$ failed to be embedded, it would either have an isolated
non-immersed point or an isolated self-intersection, the existence of either
of which would imply the existence of a non-immersed point or of a
self-intersection of the embedded curves $C^N$ for $N$ sufficiently large.
Thus, $C_0$ is an embedded curve asymptotic to a $(2,0)$ geodesic.

We now claim no such plane exists. Applying  automatic
transversality \cite{wendl} to the problem of finding nearby planes with
the same asymptotic limits, we see that $C_0$ comes in a local $1$-parameter family,
with precisely one member asymptotic to each geodesic in the $S^1$ family.
But blowing down the exceptional divisor would give us a $1$-parameter
family of planes in $\CC ^2 \setminus L$ asymptotic to $(2,0)$ geodesics
and all intersecting at a point. But such planes
have intersection number $0$, indeed, such planes are homologous relative to their boundaries to double covers of disks $D(1) \times \{p\}$ in the boundary of the polydisk. Hence we have a contradiction to positivity of intersection.


Excluding Case $2$ completes the proof of Proposition \ref{foliationF}.
\end{proof}

The following lemma gives further structure of $\mathcal F$.

\begin{lemma} \label{L:uniquePlane}
	The foliation $\mathcal F$ contains a unique leaf of degree $1$
	asymptotic to each $(-1,0)$ geodesic and a unique leaf of
	degree $0$ asymptotic to each $(1,0)$ geodesic.
\end{lemma}
\begin{proof}
	Note that the curves $C_0$ and $C_1$ both satisfy automatic
	transversality \cite{wendl}*{Theorem 1}.
Their respective moduli spaces are compact since no further breaking can occur
for energy reasons, and thus these moduli spaces consist of a collection of
circles. The kernel of their linearized operators is $1$ dimensional, and must
project non-trivially in the direction tangent to the space of Morse-Bott
orbits, so the asymptotic evaluation map must be surjective.

It remains to show that no two leaves are asymptotic to the same orbit in class
$(1,0)$ (respectively $(-1,0)$). We consider the case of the $(1,0)$ geodesic ---
the same argument applies for the $(-1,0)$ geodesic.

Suppose that both $u$ and $v$ are distinct
degree $0$ planes asymptotic to the same orbit
in the class $(1,0)$. These curves do not intersect as they are leaves in the
foliation. By the asymptotic convergence formula in Theorem \ref{T:asymptotic},
the two curves can each be represented near their ends as a graph over the asymptotic limit.
The difference between them is then of the form
\[
U(s,t) - V(s,t) = \e^{\lambda s}( e(t) + r(s,t)),
\]
where $\lambda$ is a positive eigenvalue of the asymptotic operator associated
to the closed Reeb orbit.
As computed in Appendix \ref{A:RiemannRoch}, the eigenvectors corresponding
to positive eigenvalues have winding number at least one with respect
to the framing given by the $S^1$ family of orbits. Pushing $u$ off
in the constant direction then introduces an additional intersection
point, making $u \star v = 1$. The homotopy invariance, however, gives $u \star
v = 0$. Thus, no two planes in the family can asymptote to the same orbit.

\end{proof}

\subsection{Limiting curves in $T^* L$ and $\R \times \Sigma$.} \label{tls}

\begin{proposition} \label{otherlimits}
Limits of the $C^N_i$ in $T^* L$ and $\R \times \Sigma$ are all cylinders which either have positive ends asymptotic to Reeb orbits of type $(1,0)$ and $(-1,0)$, or lie in $\R \times \Sigma$ and are trivial cylinders over an orbit of type $(1,0)$ or $(-1,0)$.

\end{proposition}

\begin{proof}
As our limiting building has genus $0$ and, by Proposition \ref{foliationF} (in the case when the limit is not a closed curve), contains exactly two curves in $X \setminus L$ asymptotic to orbits of type $(1,0)$ and $(-1,0)$, the limiting components in $T^* L$ and $\R \times \Sigma$ must fit together along their matching ends to form a cylinder with ends asymptotic to $(1,0)$ and $(-1,0)$ orbits.

As there are no contractible orbits, this implies that all curves in $T^* L$ and $\R \times \Sigma$ are cylinders and every asymptotic limit is of the type described. The only possibilities are then as described in the proposition.
\end{proof}

Note that the nature of the curves in $\R \times \Sigma$ and $T^*L$
 does not require genericity of the almost complex
structure $J$. Therefore we can choose a $J$ as in \cite{mike} so as to conclude the following.

\begin{corollary} \label{symmetricfoliation} (\cite{mike}, section $10$.)

The almost complex structure on $T^*L$ may be chosen so that the limiting
cylinders in $T^*L$ and $\R \times \Sigma$ have both ends asymptotic to
a lift of the same geodesic in $T^2 = L$.

\end{corollary}

%
%

\begin{remark}
	The construction of this foliation depended quite heavily on
	our assumption that the embedding of $P(1,2)$ is in $CP^2(R)$
	for $R < 3$. If we considered $R > 3$, many other buildings could
	arise in the foliation. In particular, we would expect to see an $S^1$
	family of buildings with levels in $X \setminus L$ consisting
	of a degree $1$ plane asymptotic to a
	$(-1,-1)$ geodesic and a degree $0$ plane intersecting $E$,
	asymptotic to a $(1,1)$ geodesic.
	The key idea of our argument is that the non-existence of these
	curves is what provides the contradiction later.
\end{remark}

\section{Degree $d$ curves with point constraints} \label{highdegree}

In this section, to complete a proof by contradiction of Theorem \ref{main}, we will consider the SFT limits of curves of degree $d$ as we
stretch the neck under the assumption that $R < 3$.

By Proposition \ref{highdeg}, to any set of $2d$ points in sufficiently general position $\{q_i\} \in X$
there exists a unique embedded $J^N$-holomorphic curve $\highdeg^N$ of degree $d$ and with $e
= \highdeg^N \bullet E = d-1$ intersecting each of the $2d$ points $\{q_i\}$. We
also observe that $\highdeg^N \bullet F =1$. Let
us fix the $2d$ points in a the small Weinstein neighbourhood $U$ of $L$. As the cylinders described by Proposition \ref{otherlimits} have deformation index at most $2$, we may assume that each such cylinder intersects at most one of the points.

We now take the limit as $N \rightarrow \infty$. As described in section \ref{indices} the limiting holomorphic
building will have components in $X \setminus L$, in $\R \times \Sigma$
and in $T^*L$.
Denote the building by $F$ and let $F_0$ denote the components of the building
in $X \setminus L$.

We will study $F_0$ using the foliation $\mathcal F$ described by Proposition \ref{foliationF}. A generic curve in the foliation is a closed curve in the fiber class, but such curves can converge to {\it broken configurations}, that is, the union of a curve of type (2) and (3) from Proposition \ref{foliationF} asymptotic to the same geodesic with opposite orientations, see Corollary \ref{symmetricfoliation}.

\begin{lemma} \label{outside}
	All components of $F_0$ except for one are (possibly branched) covers
	of curves in the foliation $\mathcal F$.
	Denote by $f_0 \subset F_0 \subset F$ the component that is not a cover.
	
For each closed curve $C$ in $\mathcal F$,
	the intersection number $f_0 \bullet C = 1$.
	For the broken configurations in  $\mathcal F$, we have
	$f_0 \star C_0 + f_0 \star C_1 = 1$.

	Furthermore, the component $f_0$ is of some degree $d'$ with
	$1 \le d' \le d$ and has intersection number with $E$ given by
	$e'=f_0 \bullet E = d' -1$.
\end{lemma}
\begin{proof}

Fix a component $f$ of the building. For a closed curve $C$ of $\mathcal F$ the intersection number $C \bullet f$ is well defined. We now consider the intersection
number $C' \bullet f$ as we vary $C'$
through the moduli space of curves in $\mathcal F$.
These intersection
numbers $C' \bullet f$ are locally constant among all deformations of $C'$
within the foliation, including through broken configurations, as long as
the breaking occurs at orbits geometrically distinct from the asymptotic
orbits of $f_i$. Broken configurations have codimension $1$ and broken configurations along an asymptotic limit of $f$ have codimension $2$ in the $2$-dimensional space of leaves of $\mathcal F$.
Hence, for any closed curve $C \in \mathcal F$, $C \cdot f$
does not depend on the choice of curve $C$, and for a broken configuration with asymptotics disjoint from those of $f$ and close to a closed curve $C$ we have
\begin{equation} \label{E:intersection1}
	f \bullet C_0 + f \bullet C_1 = f \bullet C.
\end{equation}

Let now $C_0 \cup C_1$ be a broken configuration with asymptotic ends coinciding
with the ends of a component $f$. We then have
$f \bullet C_0 + f \bullet C_1 \le f \bullet C$, where $C$ is a nearby
closed curve in $\mathcal F$. However, for the Siefring-Wendl
 extended intersection number (Theorem \ref{T:intersectionNumber}),
we have
\begin{equation} \label{E:intersection}
	f \star C_0 + f \star C_1 = f \star C = f \bullet C.
\end{equation}
In order to keep the exposition self-contained, we now briefly explain
why this identity holds for the extended intersection number.
Recall from the discussion following Theorem \ref{T:intersectionNumber},
that in our case with orbits in $\mathbb{T}^3$, the
extended intersection number
is obtained by perturbing $C_0$ and $C_1$ off of their asymptotic limits
in the direction of the $S^1$ family of closed Reeb orbits.
Now, the planes $C_0$ and $C_1$ satisfy the automatic transversality criterion
of \cite{wendl}. It follows that in their moduli spaces the asymptotic limits move in the $S^1$
Morse-Bott family of Reeb orbits in class $(1,0)$ or $(-1,0)$, and thus
the push-off of $C_0$ and of $C_1$ can be realized by holomorphic planes
$C_0'$ and $C_1'$, with asymptotic limits disjoint from those of $f$. Thus the result follows from the case of broken configurations with asymptotics disjoint from $f$.

Given these preliminaries, we proceed to prove Lemma \ref{outside}.
Denote the components of the building $F_0$ by
$f_0, f_1, \dots, f_K$.

As $F$ is a limit of curves $\highdeg^N$ having intersection number $1$ with fibre curves, if
$C$ is
a closed curve in $\mathcal F$ we have the intersection number
\[
	C \bullet f_0 + C \bullet f_1 + \cdots + C \bullet f_K = 1.
\]

By positivity of intersection each of the terms on the left here is a nonnegative integer, and so we conclude that exactly one component, which we call $f_0$, has $C \bullet f_0=1$ and all other components have $C \bullet f_i=0$.

By the analysis above, for $i \ge 1$ we have $C \bullet f_i=0$ for all curves $C$ in $\mathcal F$, including broken configurations. As $\mathcal F$ is a foliation, $f_i$ cannot be disjoint from all curves and so we conclude by positivity of intersection that it must cover one of them.
The intersection numbers for $f_0$ also follow from the above analysis.


Finally, by Proposition \ref{foliationF} we note that each curve in $\mathcal F$ has the same intersection
number with $E$ as its degree. This then holds for any cover of a curve in
$\mathcal F$, so for each $i=1, \dots, K$, $f_i \bullet \infline= f_i
\bullet E$. The high degree curves $\highdeg_N$ have
$\highdeg_N \bullet \infline = \highdeg_N \bullet E + 1$, and thus as intersection numbers are preserved in the limit we must have:
\begin{align*}
f_0 \bullet \infline &= f_0 \bullet E + 1 \\
d &= f_0 \bullet \infline + f_1 \bullet \infline + \dots + f_K \bullet \infline.
\end{align*}
The result now follows from the non-negativity of each of these intersections.

\end{proof}

As in Lemma \ref{outside}, $f_1, f_2, \dots, f_K$ denote the components of $F_0$ which cover leaves of $\mathcal F$. Then their
asymptotic limits must be closed Reeb orbits corresponding to geodesics
in class $(a, 0)$ for $0 \ne a \in \Z$. As our building has genus $0$ and there are no null-homologous
Reeb orbits,  the homology classes of the asymptotic limits of $f_0$ must match with the asymptotic limits of the other $f_i$ up to orientation and so
must all be in classes $(a,0)$ also.

\begin{lemma} \label{allsame}
Either all punctures of $f_0$ are asymptotic to geodesics in classes $(k,0)$ for $k >0$, or all punctures are asymptotic to geodesics in classes $(k,0)$ for $k <0$.
\end{lemma}

\begin{proof}
Suppose that $f_0$ has a puncture asymptotic to a closed Reeb orbit in class
$(k,0)$, $k > 0$. Denote this orbit by $\gamma^k$.
Then by Lemma \ref{L:uniquePlane} there exists a plane in the family $C_0$ asymptotic
to the underlying simple cover $\gamma$ of the same orbit. We will now estimate
$f_0 \star C_0$ using this specific plane.
Note that $f_0$ is somewhere injective since its intersection numbers
with $\infline$ and with $E$ are coprime, see Lemma \ref{outside}.
The plane $C_0$ is also somewhere injective, and the two curves have geometrically distinct
images. Therefore, $f_0 \bullet C_0$ is well-defined and non-negative. By the
asymptotic convergence formula for holomorphic curves
in Theorem \ref{T:asymptotic},
the two curves do not have any intersection near the puncture.
Furthermore, near their asymptotic limits, the curves $f_0$ and the $k$-fold cover
of the restriction of $C_0$ to a neighbourhood of $\gamma$
can be written as graphs
over $\gamma^k$.
In other words, there exist maps
$U_i \colon (-\infty, R] \times S^1 \to (\gamma)^*\xi, i=1,2$ with
  $U_i(s,t) \in \xi(\gamma^k(Tt))$ so that the images of $f_0$ and of $C_0$ restricted
  to a sufficiently small neighbourhood of their punctures, are described
  respectively by
\[
	 (-Ts, \exp_{\gamma^k(T t)} U_i(s,t)), i=1,2.
\]
Since the curves are geometrically distinct,
$U_1 - U_2$ does not vanish, and so
there exists
  a positive eigenvalue $\lambda > 0$ of the asymptotic operator
  $A_{\gamma^k, J}$ and a corresponding eigenvector so that
  \[
	  U_1(s,t) - U_2(s,t) = \e^{\lambda s} ( e(t) + r(s,t)),
  \]
  where $r$ decays exponentially fast.
As computed in Appendix \ref{A:RiemannRoch}, the eigenvectors
corresponding to positive eigenvalues have winding number at least one
with respect to the framing given by the $S^1$ family of orbits. Pushing
$C_0$ off in the constant direction then introduces at least $1$
additional intersection point, giving $f_0 \star C_0 \ge 1$. By Lemma \ref{outside} this implies that $f_0 \star C_0 = 1$ and $f_0 \star C_1 =0$.

By applying the same argument, if $f_0$ has a puncture asymptotic to a
closed Reeb orbit in class $(-k,0)$, $k>0$, it follows that $f_0 \star
C_1 = 1$ and $f_0 \star C_0 =0$.

Finally, by the homotopy invariance of the extended
intersection number, either all punctures of $f_0$ asymptote to Reeb
orbits in classes $(k,0)$ for $k \ge 1$, or they all asymptote to Reeb
orbits in classes $(-k, 0)$ for $k \ge 1$ as required.
\end{proof}

For components in $T^* L$ and $\Sigma \times \RR$ we have the following.

\begin{lemma} \label{inside}
The components of $F$ in $T^* L$ and $\Sigma \times \RR$ are all multiple covers of cylinders asymptotic to orbits of type $(1,0)$ or $(-1,0)$ as described in Proposition \ref{otherlimits}.
\end{lemma}

\begin{proof}
We argue by contradiction and suppose that a component $F_1$ of $F$ does not cover such a cylinder. Then as the cylinders foliate $T^* L$ and $\Sigma \times \RR$ we can find a cylinder $C$ which intersects $G$ nontrivially. Further, we may assume that the asymptotic limits of $C$ are disjoint from the asymptotic limits of the $f_0$ from Lemma \ref{outside}. Now fix a point on $C$ and take a limit of the $J^N$ holomorphic spheres $Q^N$ in the fiber class passing through $p$. Denote the limiting building by $G$. By Proposition \ref{otherlimits} one component of the limit in $T^* L$ or $\Sigma \times \RR$ coincides with $C$, and other components form a broken configuration in $X \setminus L$ with asymptotic limits distinct from $f_0$. By Lemma \ref{outside} the curve $f_0$ has intersection number $1$ with this broken configuration.

The building $F$ now has two isolated intersections with $G$, one on $F_1$ and the other on $f_0$. By positivity of intersection there are then at least two isolated intersections between $\highdeg^N$ and $Q^N$ for large $N$, and this is a contradiction as $\highdeg^N$ has intersection number $1$ with fiber curves.
\end{proof}

\noindent{\bf Completion of the proof of Theorem \ref{main}.}

By Lemma \ref{inside} the components of $F$ in $T^* L$ and $\Sigma \times \RR$ consist only of covers of cylinders and by the choice of our point constraints $\{ q_i \}$ at the start of this section there exist at least $2d$ such cylinders. Suppose that such a cylinder has ends which are not asymptotic to any end of $f_0$. Then they are asymptotic to the ends of $f_i$ for $i \ge 1$, that is, asymptotic to curves of $\mathcal F$. In particular, as one end of the cylinder is asymptotic to a geodesic in the class $(-1,0)$, this end must match with a component of $F$ which is a cover of a curve of type $C_1$ from Proposition \ref{foliationF}. Suppose there are $m$ such cylinders not asymptotic to $f_0$. Then $f_0$ must have at least $2d-m$ negative ends and degree $d' \le d-m$ (since the curves of type $C_1$ asymptotic to the $m$ cylinders have total degree at least $m$).

Now, if the asymptotic
limits of $f_0$ are all of the form $(k,0)$ for $k \ge 1$, then since the
building $F$ has genus $0$ there must be at least $2d-m$ planes in $F$ that are covers
of degree $1$ planes in $\mathcal F$ and asymptotic to limits of $f_0$ with the opposite orientation. Combined with the $m$ curves asymptotic to other orbits we find $2d$ covers of curves of type $C_1$ in $F$.  The degree of the building $F$
 is only $d$, however, so this is a contradiction. Hence, by Lemma \ref{allsame} all negative ends of $f_0$ must be in classes
 $(-k,0)$, and we must have $f_0 \star C_0 = 0$
and $f_0 \star C_1 = 1$.



Let $s \ge 2d-m$ be the total number of negative ends of $f_0$.
Let $1 \le d' \le d-m$ be the intersection number of $f_0$ with the line at
infinity (i.e.~its degree).
Recall that the intersection of $f_0$
with the exceptional divisor is $f_0 \bullet E = d' - 1$.
The area of $f_0$ then satisfies
\begin{align*}
0 < \text{Area}(f_0) & \le d' R - (d'-1) - s\\
	&= d'(R-2) + 1 + d' - s \\
	&\le d(R-2) + 1 + (d-m) - (2d-m)
\end{align*}
where the inequality uses $d' \le d-m \le d$ and $s \ge 2d-m$.
Rearranging, we get
\[
R-2 \ge 1-\frac{1}{d}.
\]
Letting $d \to \infty$, this inequality gives a contradiction for any $R<3$ and our proof is complete.

\subsubsection*{Acknowledgements}

The authors would like to thank Dusa McDuff for pointing out the interest of the embedding problem for $P(1,2)$.
The second author would like to thank Klaus Mohnke,
Richard Siefring and Chris Wendl for helpful conversations, and
Dietmar Salamon for raising questions that improved the exposition.
His work was partially supported by the ERC Starting Grant of
Fr\'ed\'eric Bourgeois StG-239781-ContactMath
and by the ERC Starting Grant of Vincent Colin Geodycon.

\appendix

\section{Morse-Bott Riemann-Roch} \label{A:RiemannRoch}

First, we state the index formula for a real-linear Cauchy-Riemann-type
operator acting on sections of a complex line bundle over a punctured
Riemann surface $(\dot \Sigma, j)$ (where $\dot \Sigma = \Sigma
\setminus \Gamma$, a finite collection of punctures).  We define
$W^{1,p,\delta}(\dot \Sigma, E)$ to be the weighted space of sections
with exponential decay (for $\delta > 0$) at the punctures.
(See the discussion preceding Theorem \ref{T:asymptotic} for a review of
the definition of an asymptotic operator.)

Then, we reformulate the index theorem from \cite{MatthiasThesis}:
\begin{theorem}
\label{T:RiemannRoch} A real Cauchy--Riemann-type operator $T:
W^{1,p,\delta}(\dot \Sigma, E) \rightarrow L^{p,\delta}(\dot
\Sigma, \Omega^{0,1}(E))$ is Fredholm if and only if all of its
$\delta$--perturbed asymptotic operators are non-degenerate.

For any choice of cylindrical asymptotic trivialization $\Phi$, the index is given by:
\[
\Ind = n \chi( \dot \Sigma) + \sum_{z \in \Gamma_{+}} \CZ^{\Phi}( A_{z} + \delta_{z} ) -
\sum_{z \in \Gamma_{-}} \CZ^{\Phi}(A_{z} - \delta_{z}) + 2 c^{\Phi}_{1}( E ),
\]
where $n$ denotes the (complex) rank of $E$, $\chi(\dot \Sigma)$
is the Euler characteristic of the punctured surface, $\CZ(A)$
is the Conley-Zehnder index  of the asymptotic operator $A$, and
$c_{1}^{\Phi}(E)$ is the first Chern class of the bundle $E$ relative to
the asymptotic trivialization data $\Phi$, paired with the fundamental
class of $[\Sigma]$.

In particular then, for a punctured holomorphic curve $u$ in an almost complex
symplectic manifold $(V, J)$
with symplectization ends,
considered modulo domain reparametrizations, the index is given by
\[
	\Ind(u) = - \chi(\dot \Sigma) +
	\sum_{z \in \Gamma_{+}} \CZ^{\Phi}( A_{z} + \delta_{z} ) -
	\sum_{z \in \Gamma_{-}} \CZ^{\Phi}(A_{z} - \delta_{z})
	+ 2 c^{\Phi}_{1}( u^*TV )
\]
where, now, the $\Phi$ are given by trivializations of the contact structure
over the asymptotic limits of $u$ and $c^{\Phi}_{1}( u^*TV )$ is, as before,
the first Chern class of the bundle relative to the asymptotic trivialization
given by $\Phi$, paired with the fundamental class of $[\Sigma]$.

\end{theorem}

Recall that for a 2-dimensional contact structure, there is an alternative
description of the Conley-Zehnder index in terms of winding numbers of
eigenvectors of the corresponding asymptotic operators \cite{hofi}.

In $T^3$, the contact structure is globally trivial.
Given explicitly, writing $T^3 = (\R / 2\pi \Z)^3$ with circle valued
coordinates $\theta, q_1, q_2$, we have $\alpha = \cos(\theta) dq_1 + \sin(\theta) dq_2$ and a trivialization of $\xi$ by $\sin(\theta) \partial_{q_1} - \cos(\theta) \partial_{q_2}$ and $\partial_\theta$.
With respect to this trivialization, the asymptotic operator
at a $T$-periodic orbit is
\[
	- J_0 \frac{d}{dt} - T \begin{pmatrix} 0 & 0 \\ 0 & 1 \end{pmatrix}
\]


We note that the spectrum includes $0$ and $-T$. Each of these
eigenvalues has multiplicity $1$ and has eigenvectors
of winding $0$ with respect to our trivialization.  The perturbed operator
corresponding to a negative puncture allowed to move in the Morse-Bott
family is the one coming from allowing asymptotic exponential growth.
This then gives an even orbit and thus Conley-Zehnder index $0$.
Similarly, for a positive puncture allowed to move in the Morse-Bott family,
the Conley-Zehnder index is then $1$.

For a curve contained in $T^*T^2$ or in $\R \times \Sigma$, the Chern number of
$u^*T$ with respect to this trivialization is $0$, since our chosen
trivialization extends.

We now need to compute $c_1(u^*TX)$ for a punctured sphere in $X \setminus
L$, of degree $d$ and of intersection $e$ with the exceptional divisor,
and asymptotic ends on orbits in classes $(k_i, l_i)$.  Observe that
this is additive with respect to connect sums, so we only need to compute
this for a curve with $d=e=0$, since the first Chern number of $TX|_{\CC
P^1(\infty)}$ is $3$ and the first Chern number of $E$ is $1$ (and a
closed curve of degree $d$ and intersection $e$ with $E$ represents
the class $d[\CC P^1] - e[E]$).  Any such curve can now be deformed
and decomposed to a collection of $s$ disks of degree $0$ and $e=0$,
each with boundary on a geodesic of class $(k_i, l_i)$, $i=1, \dots,
s$, and thus these may be taken to be model disks contained in $\C^2$
with boundary on $L$.  Finally, this trivialization of $\xi$ induces
a trivialization of $TT^*T^2|_{\Sigma}$, and thus of $\Lambda^2_\C
TT^*T^2|_{\Sigma}$. A straightforward computation shows gives that this
trivialization has winding number $k_i+l_i$ with respect to the obvious
trivialization of $\Lambda^2 T\C^2$, and thus the relative Chern number
is $k_i+l_i$, as claimed.


\end{document}